\documentclass[10pt]{article}
\usepackage[a4paper, total={4.5in, 7.2in}]{geometry}
\usepackage{amsmath,amsthm,amsfonts,amssymb}
\usepackage{mathptm}
\usepackage[pdftex,pdfborder={0 0 0}]{hyperref}
\usepackage{multirow}
\usepackage{array}
\usepackage{bbm}
\usepackage[noblocks]{authblk}
\usepackage{verbatim}
\usepackage{tikz}
\usepackage{float}
\usepackage{enumerate}
\usepackage{diagbox}

\usepackage{amsfonts}
\usepackage{amssymb}
\usepackage{amsmath}
\usepackage{amsthm}
\usepackage{graphicx}
\usepackage{eucal}
\usepackage{tikz}
\usepackage{float}
\usepackage{subcaption}
\usepackage{csquotes}
\usepackage{amsmath}
\usepackage{amsthm}
\usepackage{amssymb}
\usepackage{amsrefs}
\usepackage{amsfonts, mathtools}

\newtheorem{theorem}{Theorem}[section]
\newtheorem{lemma}[theorem]{Lemma}
\newtheorem{proposition}[theorem]{Proposition}
\newtheorem{corollary}[theorem]{Corollary}

\DeclarePairedDelimiter{\ceil}{\lceil}{\rceil}
\DeclarePairedDelimiter{\floor}{\lfloor}{\rfloor}

\newcommand{\hidethis}[1]{}

\theoremstyle{definition}
\newtheorem{definition}[theorem]{Definition}

\title{Extremal Problems Related to the Cardinality-Redundance of Graphs}

\author[1]{Daniel McGinnis}
\author[2]{Nathan Shank}

\affil[1]{New College of Florida (daniel.mcginnis15@ncf.edu)}
\affil[2]{Moravian College (shankn@moravian.edu)}
\date{}
\begin{document}
\maketitle

\begin{abstract}
A dominating set of a graph $G$ is a set of vertices $D$ such that for all $v \in V(G)$, either $v \in D$ or $[v,d] \in E(G)$ for some $d \in D$. The cardinality-redundance of a vertex set $S$, $CR(S)$, is the number of vertices $x\in V(G)$ such that $|N[x] \cap S| \geq 2$. The cardinality-redundance of $G$ is the minimum of $CR(S)$ taken over all dominating sets $S$. A set of vertices $S$ such that $CR(S)=CR(G)$ is a $\gamma_{CR}$-set, and the size of a minimum $\gamma_{CR}$-set is denoted $\gamma_{CR}(G)$. Here, we are concerned with extremal problems concerning cardinality-redundance. We give the maximum number of edges in a graph with a given number of vertices and given cardinality-redundance. In the cases that $CR(G)=0$ or $1$ we give the minimum and maximum number of edges of graphs when $\gamma_{CR}(G)$ is fixed, and when $CR(G)=2$ we give the maximum number edges of graphs where $\gamma_{CR}(G)$ is fixed. We give the minimum and maximum values of $\gamma_{CR}(G)$ when the number of edges are fixed and $CR(G)=0,1$, and we give the maximum values of $\gamma_{CR}(G)$ when the number of edges are fixed and $CR(G)=2$.

\end{abstract}

\section{Introduction}

Assume throughout that $G$ is a simple graph with vertex set $V(G)$ and edge set $E(G)$. Two vertices $u,v\in V(G)$ are said to be adjacent, or connected, if $[u,v]\in E(G)$. The \textit{open neighborhood} of a vertex $u\in V(G)$, denoted by $N(u)$, is defined as $N(u)=\{v\in V(G)\ :\ \textrm{$v$ is adjacent to $u$}\}$, and the \textit{degree}, $deg(u)$, is $deg(u)=|N(u)|$. The \textit{closed neighborhood} of $u$, $N[u]$, is defined as $N[u]=N(u)\cup \{u\}$.  If $v\in N[u]$, we will say that $u$ \textit{dominates} $v$ or $v$ is \textit{dominated} by $u$. Similarly, for a set $D \subseteq V(G)$, the \textit{open neighborhood} of $D$ is $N(D)=\cup_{d\in D}N(d)$ and the \textit{closed neighborhood} of $D$ is $N[D]=N(D)\cup D$. A vertex $v\in V(G)$ is said to be a \textit{private neighbor} a vertex $d$ in a set $D$ if $N[v]\cap N[D]=\{d\}$, $v$ is a \textit{shared neighbor} of $D$ if $|N[v]\cap N[D]|\geq 2$. First defined by Slater and Grinstead \cite{grinsteadfractional}, the \textit{influence} of $D$ is $I(D)=\sum_{v\in D} |N[v]|$ which is equivalent to $I(D) = \sum_{v\in V(G)} |N[v]\cap D|$. A set $S\subseteq V(G)$ is said to be a \textit{dominating set} if $N[S]=V(G)$. A comprehensive survey of domination in graphs can be found in \cite{fundamentalshaynes}. 

A vertex $u$ is said to be \textit{overdominated by $S$} if $|N[u]\cap S|\geq 2$, and if $S$ is understood from context, we will just say that $u$ is overdominated. The \textit{cardinality-redundance of $S$}, $CR(S)$, is the number of vertices of $G$ that are overdominated by $S$. The \textit{cardinality-redundance of $G$}, $CR(G)$, is the minimum of $CR(S)$ taken over all dominating sets $S$. Here, we define the parameter $\gamma_{CR}(G)$ to be the minimum size of a dominating set $S$ that satisfies $CR(S)=CR(G)$; $S$ is then called a $\gamma_{CR}$-set.  The concept of cardinality-redundance was first introduced in \cite{johnsonmaximum}, in which the authors were primarily concerned with the computation complexity of finding the minimum possible cardinality-redundance of a maximum independent set in general graphs and in series-parallel graphs. In this paper, we are more concerned about the structure of graphs satisfying certain properties related to cardinality redundnace, and constructions of graphs achieving certain extremal conditions related to cardinality-redundance.

Many extremal problems related to various domination parameters have been studied in the literature; the papers  \cite{asplundsome}, \cite{dankelmannmaximum}, \cite{desormeauxextremal}, \cite{duttonextremal}, \cite{fischermannmaximum}, \cite{henninggraphs}, and \cite{joubertmaximum} are a few examples.
In this paper we consider extremal problems related to the parameters $CR(G)$, $\gamma_{CR}(G)$, $|V(G)|$, and $|E(G)|$ when some are varied and some are fixed. In particular, we are interested in finding the maximum and minimum size of a graph $G$ when the values $|V(G)|$, $CR(G)$, and $\gamma_{CR}(G)$ are fixed. We find these values when $CR(G)=0$ or $1$, and we find the maximum size of $G$  when $CR(G)=2$. We also consider the maximum and minimum values of $\gamma_{CR}(G)$ when $|V(G)|$, $CR(G)$ and $|E(G)|$ are given. We find these values when $CR(G)=0$ or $1$, and we find the maximum of $\gamma_{CR}(G)$ when $CR(G)=2$. Furthermore, we find the maximum size of $G$ when $CR(G)$ is arbitrary, $|V(G)|$ is given, and $\gamma_{CR}(G)$ is allowed to vary.

\section{Definitions and Preliminary Results}

Throughout this paper, we will assume $n$ and $r$ are positive integers, and $m$ and $k$ will be a nonnegative integers. We will primarily be concerned with the values presented in Definitions \ref{def1}-\ref{def4} below.  We will use the standard notation and denote the set of all graphs of order $n$ and size $m$ by $G(n,m)$.

\begin{definition}\label{def1}
For $k, r, $ and $n$ with $k \leq n$ and $r \leq n$, let 
\[
M(n,k,r)=max\{|E(G)|\ :\ |V(G)|=n,\ CR(G)=k,\ \textrm{and } \gamma_{CR}(G)=r\}.
\]

If there does not exist a graph $G$ such that $|V(G)|=n$, $CR(G)=k$, and $\gamma_{CR}(G)=r$, then $M(n,k,r)=0$.
\end{definition}

Thus $M(n,k,r)$ is the maximum size of a graph $G$ of order $n$ where $CR(G) = k$ and $\gamma_{CR}(G) = r$. 

\begin{definition}
For $k, r, $ and $n$ with $k \leq n$ and $r \leq n$, let 
\[
m(n,k,r)=min\{|E(G)|\ :\ |V(G)|=n,\ CR(G)=k,\ \textrm{and } \gamma_{CR}(G)=r\}.
\]
If there does not exist a graph $G$ such that $|V(G)|=n$, $CR(G)=k$, and $\gamma_{CR}(G)=r$, then $m(n,k,r)=0$.
\end{definition}

Thus $m(n,k,r)$ is the minimum size of a graph $G$ of order $n$ where $CR(G) = k$ and $\gamma_{CR}(G) = r$.

\begin{definition}
For $k\leq n$ and $m \leq {n \choose 2}$, let
\[
D(n,k,m)=max\{\gamma_{CR}(G)\ :\ G \in G(n,m),\ \textrm{and}\ CR(G)=k\}.
\]
If there does not exist a graph $G$ such that $G\in G(n,m)$ and $CR(G)=k$, then $D(n,k,m)=0$.
\end{definition}

Thus $D(n,k,m)$ is the maximum value of $\gamma_{CR}(G)$ over all graphs with order $n$ and size $m$ whose cardinality-redundance is $k$.  

\begin{definition}\label{def4}
For $k\leq n$ and $m \leq {n \choose 2}$, let
\[
d(n,k,m)=min\{\gamma_{CR}(G)\ :\ G \in G(n,m),\ \textrm{and}\ CR(G)=k\}.
\]
If there does not exist a graph $G$ such that $G\in G(n,m)$ and $CR(G)=k$, then $d(n,k,m)=0$.
\end{definition}

Thus $d(n,k,m)$ is the minimum value of $\gamma_{CR}(G)$ over all graphs with order $n$ and size $m$ whose cardinality-redundance is $k$.  So $D(n,k,m)$ is a maximum of minimum over a set and $d(n,k,m)$ is a minimum of minimum over the same set. 

We now give some simple results on $\gamma_{CR}$-sets and cardinality-redundance which will help us bound $CR(G)$. Note that a dominating set $S$ is minimal if for every vertex $s \in S$ the set $S \setminus\{s\}$ is not a dominating set.  

\begin{proposition}\label{minimalgamma}
Let $G$ be a graph. Every $\gamma_{CR}$-set of $G$ is a minimal dominating set.
\end{proposition}

\begin{proof}
Let $S$ be a $\gamma_{CR}$-set of $G$. If $S$ is not minimal, then there exists some $s\in S$ such that $S\setminus \{s\}$ is a dominating set of $G$. Clearly $CR(S\setminus \{s\})\leq CR(S)$. Since $|S\setminus \{s\}|<|S|$, this contradicts that $S$ is a $\gamma_{CR}$-set. Therefore, $S$ is a minimal dominating set.

\end{proof}

\begin{proposition}\label{gammabnd}
Let $G$ be a graph of order $n$. If $S$ is a minimal dominating set and $CR(S)=k$, then $|S|\leq n-k$.
\end{proposition}

\begin{proof}
Assume $S$ is a minimal dominating set and $CR(S)=k$. Then there is a set $D$ of $n-k$ vertices of $G$, where if $u\in D$, then $|N[u]\cap S|=1$, i.e., $u$ is the private neighbor of some element of $S$. Since $S$ is minimal, each vertex contained in $S$ has at least one private neighbor. Therefore, $|S|\leq n-k$.

\end{proof}

The following proposition gives a lower bound for the influence of a set based on the cardinality redudance. 

\begin{proposition}\label{infbnd}
Let $G$ be a graph of order $n$ and let $S$ be a dominating set of $G$. If $CR(S)=k$, then $I(S)\geq n+k$.
\end{proposition}

\begin{proof}
Let $S$ be a dominating set for $G$ and let $D = \{u \in V(G):|N[u]\cap S|=1\}$. Since $CR(S)=k$, $|D|=n-k$. Therefore, 
\begin{align*}
I(S)&=\sum_{x\in V(G)}|N[x]\cap S|\\
&=\sum_{x\in V(G)\setminus D}|N[x]\cap S|+\sum_{x\in D}|N[x]\cap S|\\
&=\sum_{x\in V(G)\setminus D}|N[x]\cap S|+(n-k)\\
&\geq 2k+(n-k)=n+k.
\end{align*}

\end{proof}

The following proposition shows if every dominating set overdominates at least one vertex ($CR(G) \geq 1$), then the size of a $\gamma_{CR}$-set must be at least two.  

\begin{proposition}\label{gammageq2}
If $G$ is a graph such that $CR(G)\geq 1$, then $\gamma_{CR}(G)\geq 2$.
\end{proposition}

\begin{proof}
If $S$ is a $\gamma_{CR}$-set of $G$, then there exists some $x\in V(G)$ such that $|N[x]\cap S|\geq 2$. This implies that $|S|\geq 2$.

\end{proof}

One the other hand, if $S$ is a $\gamma_{CR}$-set, then $S$ overdominates at most $n-2$ vertices.  

\begin{proposition}\label{crbnd}
If $G$ is a graph of order $n \geq 2$, then $CR(G)\leq n-2$.
\end{proposition}

\begin{proof}
Let $S$ be a $\gamma_{CR}$-set of $G$. If ${CR}(G)=0$, then we are done. If ${CR}(G)=k\geq 1$ then by Proposition \ref{gammageq2}, $|S|\geq 2$. Therefore, by Propositions \ref{minimalgamma} and \ref{gammabnd}, $k\leq n-2$.

\end{proof}

The following theorem characterizes all the graphs whose cardinality-redundance achieves the bound given in Proposition \ref{crbnd}.

\begin{theorem}\label{eventhm}
Let $G$ be a graph of order $n\geq 4$ vertices. Then $CR(G)=n-2$ if and only if $n$ is even and each vertex of $G$ has degree $n-2$.
\end{theorem}

\begin{proof}
Let $n\geq 4$ be an even integer and assume each vertex of $G$ has degree $n-2$. Let $S$ be $\gamma_{CR}$-set. By Propositions \ref{minimalgamma} and \ref{gammabnd}, $S$ is a minimal dominating set and $|S|=2$.  Let $S=\{s_1,s_2\}$, where $s_1$ and $s_2$ are not adjacent to some $s_1'\neq s_1$ and $s_2'\neq s_2$, respectively. Since $S$ is a dominating set $s_1'\neq s_2'$. Each vertex contained in $V(G)\setminus \{s_1',s_2'\}$ is overdominated by $S$, so $CR(G)=n-2$.

Let $CR(G)=n-2$, and let $W=\{v\in V(G)\ :\ deg(v)\leq n-3\}$. Assume $W$ is non-empty. Since $n-2\neq 0$, $G$ has no vertex with degree $n-1$. If $W$ is not a dominating set, then there exists some $v\in V(G)$ such that $deg(v)=n-2$ and is adjacent to no vertex in $W$. Therefore, $W=\{u\}$ for some $u \in V(G)$, and $\{v,u\}$ is a dominating set of $G$. However, since $v$ is not adjacent to $u$ and $deg(u)\leq n-3$, $CR(\{u,v\})\leq n-3$, contradicting that $CR(G)=n-2$. If $W$ is a dominating set, then there exists some $W' \subseteq W$ such that $W'$ is a minimal dominating set. Since $CR(G)=n-2$, $CR(W')\geq n-2$, and by Propositions \ref{gammabnd} and \ref{gammageq2}, $|W'|=2$. By definition of $W$, $I(W')\leq 2(n-2)<2n-2$, but this contradicts Proposition \ref{infbnd}. Therefore, $W$ is empty, so every vertex of $G$ has degree $n-2$ which is possible if and only if $n$ is even.

\end{proof}

\begin{corollary}
Let $G$ be a graph of order $n$. If $n$ is odd, then $CR(G)\leq n-3$.
\end{corollary}

\begin{proof}
This follows from Proposition \ref{crbnd} and Theorem \ref{eventhm}.

\end{proof}

\section{Results for  $CR(G)=0$}

In this section we find the exact values of $M(n,0,r)$, $m(n,0,r)$, $D(n,0,m)$, and $d(n,0,m)$. If $CR(G)=0$ then $G$ has an \textit{efficient dominating set} which were first studied by Biggs \cite{Biggs1973} as a perfect code and then by Bange, Barkauskas, and Slater as efficient dominating sets  (\cite{Bange1996}, \cite{Bange1988}, \cite{Bange1987}).  We will start by considering $M(n,0, r)$. 

\begin{theorem}\label{M(n,0,r)}
For any positive integers $2\leq n$, and $1 \leq r \leq n-1$
\[
M(n,0,r)={n-r+1 \choose 2}.
\]
\end{theorem}

\begin{proof}
Let $G$ be a graph such that $|V(G)|=n$, $CR(G)=0$, and $\gamma_{CR}(G)=r$, and let $S$ be a $\gamma_{CR}$-set of $G$. Since $CR(G)=0$, each vertex in $V(G)\setminus S$ is adjacent to precisely one vertex in $S$. Thus, there are precisely $n-r$ edges between $S$ and $V(G)\setminus S$. There are at most ${n-r \choose 2}$ edges between the vertices in $V(G)\setminus S$, and no two vertices in $S$ can be adjacent. Therefore $|E(G)|\leq (n-r) + {n-r \choose 2}={n-r+1 \choose 2}$. Since $G$ was arbitrary, this implies that $M(n,0,r)\leq {n-r+1 \choose 2}$.

Now, consider a graph $\Gamma$ that is a complete graph on $n-r+1$ vertices unioned with $r-1$ isolated vertices.  Notice that $\Gamma$ satisfies $|V(\Gamma)|=n$, $CR(\Gamma)=0$, and $\gamma_{CR}(\Gamma)=r$. Therefore $M(n,0,r) \geq {n-r+1 \choose 2}$. This combined with the above give the desired result.

\end{proof}

Notice Theorem \ref{M(n,0,r)} implies that if $r\geq s$, then $M(n,0,r)\leq M(n,0,s)$. Therefore, since $M(n,0,r)={n-r+1 \choose 2}$, we have the following corollary: 

\begin{corollary}\label{cor:M(n,0,r)}
If $m > {n-r+1 \choose 2}$ for some $r$, then for a graph $G$ of order $n$ and size $m$, $CR(G)\leq r-1$. 
\end{corollary}

Now we will find the value of the minimum: $m(n,0,r)$.  

\begin{theorem}\label{m(n,0,r)}
For any positive integers $2\leq n$, and $1 \leq r \leq n-1$

\[m(n,0,r)=n-r.
\]
\end{theorem}

\begin{proof}
Let $G$ be a graph such that $|V(G)|=n$, $CR(G)=0$, and $\gamma_{CR}(G)=r$, and let $S$ be a $\gamma_{CR}$-set of $G$. Since $CR(G)=0$, each vertex in $V(G)\setminus S$ is adjacent to precisely one vertex in $S$. Thus, there are precisely $n-r$ edges between $S$ and $V(G)\setminus S$, so $|E(G)| \geq n-r$. Since $G$ was arbitrary, $m(n,0,r)\geq n-r$.

Now, consider a graph $\Gamma$ that is a star graph on $n-r+1$ vertices unioned with $r-1$ isolated vertices.  Notice that $|V(\Gamma)|=n$, $CR(\Gamma)=0$, and $\gamma_{CR}(\Gamma)=r$. Therefore, $m(n,0,r) \leq n-r$. This combined with the above gives us the desired result.

\end{proof}

Notice Theorem \ref{m(n,0,r)} implies that if $r\geq s$, then $m(n,0,r)\leq m(n,0,s)$. Therefore, since $m(n,0,r)=n-r$, we have the following corollary: 

\begin{corollary}\label{cor:m(n,0,r)}
If $m< n-r$ for some $r$, then for a graph $G$ with $n$ vertices and $m$ edges, $CR(G)\geq r+1$.
\end{corollary}

Now we will use the previous two corollaries to find $D(n,0,m)$ and $d(n,0,m)$. 

\begin{theorem}\label{D(n,0,m)}
For any positive integers $n \geq 2$ and $ 0 \leq m \leq {n \choose 2}$,  
\[
D(n,0,m)=n-r,
\]
where ${r \choose 2} < m \leq {r+1 \choose 2}$.

\end{theorem}

\begin{proof}
By Corollary \ref{cor:M(n,0,r)}, since ${r \choose 2} < m$, $D(n,0,m)\leq n-r$.

Let $R$ be a set of $r$ vertices such that its induced subgraph has $m-r$ edges. Notice it is possible for the induced subgraph on $R$ to have $m-r$ edges since $m-r \leq {r+1 \choose 2} -r = {r \choose 2}$. Now, connect a vertex $v$ to each vertex in $R$, and let $\Gamma$ be the disjoint union of this graph with $n-r-1$ isolated vertices. Then $\Gamma$ satisfies $|V(\Gamma)|=n$, $|E(\Gamma)|=m$, and $CR(\Gamma)=0$. Since $\gamma_{CR}(\Gamma)= n-r$, $D(n,0,m)\geq n-r$. This combined with the above gives us the desired result.

\end{proof}

\begin{theorem}
For any positive integers $n\geq 2$ and $0 \leq m \leq {n \choose 2}$
\[
d(n,0,m)= 
    \begin{cases}
        n-m & \quad \text{if $0 \leq m < n-1$},\\
        1 & \quad \text{if $n-1 \leq m \leq {n \choose 2}$}.
    \end{cases}
\]
\end{theorem}

\begin{proof}
If $n-1\leq m\leq {n \choose 2}$, then clearly it is possible to construct a graph of order $n$ and size $m$ where one vertex is adjacent to every other vertex. This graph would then have a cardinality-redundance of $0$, so $d(n,0,m)\leq 1$. Clearly $d(n,0,m)\geq 1$, so $d(n,0,m)=1$.

If $m < n-1$, then by Corollary \ref{cor:m(n,0,r)}, with $r=n-m-1$, we see $d(n,0,m)\geq n-m$. Consider the graph $\Gamma$ that is a star graph on $m+1$ vertices unioned with $n-m-1$ isolated vertices. Then $\Gamma$ satisfies $|V(\Gamma)|=n$, $|E(\Gamma)|=m$, and $CR(\Gamma)=0$. Since $\gamma_{CR}(\Gamma)=n-m$, $d(n,0,m)\leq n-m$. This combined with the above gives us the desired result.

\end{proof}

\section{Results for  $CR(G)=1$}

In this section we find the exact values of $M(n,1,r)$, $m(n,1,r)$, $D(n,1,m)$, and $d(n,1,m)$. The following proposition gives an upper bound on $\gamma_{CR}(G)$ when $CR(G)=1$.  

\begin{proposition}\label{prop1}
If $G$ is a graph of order $n$ and $CR(G)=1$, then $n\geq 5$, and $\gamma_{CR}(G)\leq n-3$.
\end{proposition}

\begin{proof}
The condition that $n\geq 5$ can be verified by inspection of all graphs with less than $5$ vertices.

Let $S$ be a $\gamma_{CR}$-set of $G$, and let $x$ be the vertex of $G$ overdominated by $S$. Notice that no two vertices of $S$ are adjacent, else they would both be overdominated, thus, $x \notin S$. If $V(G)=S \cup \{x\}$, then $\{x\}\cup (S\setminus N(x))$ is a dominating set that overdominates no vertices. Therefore there is some vertex $y\neq x$ such that $y \notin S$. If $V(G)=S \cup \{x,y\}$, then when either $x$ or some vertex in $S\setminus N(x)$ is adjacent to $y$, $\{x\} \cup (S\setminus N(x))$ overdominates at most the vertex $y$, but $|\{x\} \cup (S\setminus N(x))|< |S|$, contradicting that $S$ is a $\gamma_{CR}$-set of $G$. When $x$ is not adjacent to $y$ and some vertex in $N(x)$ is adjacent to $y$, then $\{x,y\}\cup (S\setminus N(x))$ is a dominating set that overdominates one vertex. Notice, however, that $G$ is the induced subgraph on $N(x)\cup \{x,y\}$ unioned with the isolated vertices $S\setminus N(x)$. Therefore, by the first part of this theorem, $|N(x)\cup \{x,y\}|\geq 5$, so $|N(x)|\geq 3$. Thus, $|\{x,y\}\cup (S\setminus N(x))|<|S|$, contradicting that $S$ is a $\gamma_{CR}$-set. This implies that $|S|\leq n-3$, i.e., $\gamma_{CR}(G)\leq n-3$.

\end{proof}

By Proposition \ref{prop1} we see that $M(n,1,r)=m(n,1,r)=0$ if $r \geq n-2$, and $D(n,1,m)$, $d(n,1,m) \leq  n-3$ for all $m$.  

For the following two theorems, we consider $2 \leq r \leq n-3$ and find the exact value of $M(n,1,r)$ and $m(n,1,r).$ Note that if $CR(G) \geq 1$, then $r \geq 2$, so when $CR(G) \geq 1$, we will only concern ourselves with $r \geq 2.$

\begin{theorem}\label{Mn1r}
Let $n\geq 5$ and $ 2\leq r\leq n-3$, then 
\[
M(n,1,r)={n-r \choose 2}+ (n-2).
\]
\end{theorem}

\begin{proof}
Let $G$ be a graph with $n$ vertices such that $CR(G)=1$ and $\gamma_{CR}(G)=r$. Let $S$ be a $\gamma_{CR}$-set, let $x$ be the overdominated vertex, and let $B=V(G)\setminus (S\cup\{x\})$. Notice $x \notin S$ if $CR(G) = 1$, therefore, $|B|=n-r-1$. Since $CR(G)=1$, no two vertices in $S$ are adjacent. There are at most $r$ edges between $S$ and $\{x\}$, and precisely $n-r-1$ edges between $S$ and $B$. There are at most $n-r-1$ edges between $\{x\}$ and $B$, and at most ${n-r-1 \choose 2}$ edges between the vertices in $B$. Also the $x$ must not be connected to at least one vertex in $S$ or $B$, otherwise $\{x\}$ would be a dominating set that overdominates no vertices. Therefore, 

\[
|E(G)|\leq r+(n-r-1)+(n-r-1)+{n-r-1 \choose 2}-1={n-r \choose 2}+(n-2).
\] This implies that $M(n,1,r)\leq {n-r \choose 2}+(n-2)$.

Now consider a graph $\Gamma$ whose vertices are the disjoint union $V(\Gamma)=S'\cup \{x'\}\cup B'$, where $|S'|=r$ and $|B'|=n-r-1$. Notice that $|B'|\geq 2$. Connect every vertex in $S'$ to $x'$, connect one vertex $s_1 \in S'$ to every vertex in $B'$ except one vertex $y\in B'$, connect $s_2 \in S'$ where $s_2\neq s_1$ to $y$, and connect $x'$ to every vertex in $B'$ except $y$. Finally, add an edge between every pair of vertices in $B'$. Let $z\in B'$ be a vertex such that $z \neq y$. Since $|B'|\geq 2$, such a vertex $z$ exists. Notice that $S'$ is a dominating set and overdominates one vertex, $x'$. The vertex $x'$ cannot be in any $\gamma_{CR}$-set of $\Gamma$ since $x'$ is adjacent to every other vertex except $y$, and if $y$ and $x'$ are both in a $\gamma_{CR}$-set, then $z$ and $s_2$ would both be overdominated. No vertex in $B'\setminus \{y\}$ can be in a $\gamma_{CR}$-set, since the only ways to dominate $s_2$ would be if $s_2$, $x'$ or $y$ is in the same $\gamma_{CR}$-set, but in each of these cases more than one vertex is overdominated. If $y$ is in some $\gamma_{CR}$-set of $\Gamma$, then by the above arguments, the only way that the vertices in $S'\setminus \{s_2\}$ can be dominated is if $S'\setminus \{s_2\}$ is in the same $\gamma_{CR}$-set as $y$. But if $|B'|>2$ or $|S'|>2$, then more than one vertex would be overdominated, so $|B'|=2$ and $|S'|=2$. In this case, the only $\gamma_{CR}$-sets of $\Gamma$ are $S'$ and $\{y,s_1\}$ ($\{y,s_1\}$ overdominates only $z$), so $CR(\Gamma)=1$ and $\gamma_{CR}(\Gamma)=r=2$. In every other case, $S'$ is the only $\gamma_{CR}$-set of $\Gamma$, so $CR(\Gamma)=1$ and $\gamma_{CR}(\Gamma)=r$. Since $|E(\Gamma)|={n-r \choose 2}+(n-2)$, $M(n,1,r)\geq {n-r \choose 2}+(n-2)$. This combined with the above gives us the desired result.

\end{proof}

\begin{theorem}\label{mn1r}
Let $n\geq 5$, and let $2\leq r\leq n-3$, then 
\[
m(n,1,r)=n-r+1.
\]
\end{theorem}

\begin{proof}
Let $G$ be a graph with $n$ vertices such that $CR(G)=1$ and $\gamma_{CR}(G)=r$. Let $S$ be a $\gamma_{CR}$-set of $G$, let $x$ be the vertex overdominated by $S$, and let $B$ be the set of $n-r-1$ vertices disjoint from $S$ and $\{x\}$. Since $x$ is dominated by $S$, there are at least two vertices in $S$ that are adjacent to $x$. There are precisely $n-r-1$ edges between $S$ and $B$. Therefore, $|E(G)|\geq 2+(n-r-1)=n-r+1$. Since $G$ was arbitrary, this implies that $m(n,1,r)\geq n-r+1$.

Now consider a graph $\Gamma$ whose vertices are the disjoint union $V(\Gamma)=S'\cup \{x'\}\cup B'$, where $|S'|=r$ and $|B'|=n-r-1$. Notice that $|B'|\geq 2$. Connect two vertices $s_1,s_2 \in S'$ to $x'$, and connect $s_1$ to each vertex in $B'$. Notice that $S'$ is a dominating set of $\Gamma$, and $CR(S')=1$. If $s_1$ is not in some $\gamma_{CR}$-set $H$, then each vertex in $B'$ must be in $H$, and in order to dominate $s_2$ and $x'$, $s_2$ or $x'$ must be in $H$. However, since we must have $S'\setminus \{s_1,s_2\}\subset H$, $|H|>|S'|$, contradicting that $H$ is a $\gamma_{CR}$-set since $H$ overdominates $s_1$. Therefore, $s_1$ is in every $\gamma_{CR}$-set and to dominate $s_2$ and $x'$ while only overdominating one vertex, $s_2$ must be in every $\gamma_{CR}$-set. Clearly, $S'\setminus \{s_1,s_2\}$ is in every $\gamma_{CR}$-set, so $S'$ is the unique $\gamma_{CR}$-set of $\Gamma$. Therefore, $CR(\Gamma)=1$ and $\gamma_{CR}(\Gamma)=r$. Since $|E(\Gamma)|=n-r+1$, we have that $m(n,1,r)\leq n-r+1$. This combined with the above gives us the desired result.

\end{proof}

Next we will consider $D(n,1,m)$ and $d(n,1,m)$ where $m$ depends on $n$. 

\begin{theorem}\label{Dn1r}
Let $n\geq 5$. Then

\[
D(n,1,m) = 
    \begin{cases}
        n-3 & \text{if $4\leq m \leq n+1$}, \\
        r & \text{if $n+2 \leq m \leq {n-1 \choose 2}$ and }\\
         & \qquad \text{${n-(r+1) \choose 2} + (n-2) < m \leq {n-r \choose 2} + (n-2)$},\\
        0 &\text{if $m > {n-1 \choose 2}$}.\\
    \end{cases}
\]
\end{theorem}

\begin{proof}
Let $n\geq 5$.  

Case 1: Assume $4\leq m \leq n+1$.

Let $G$ be a graph whose vertices are the disjoint union $V(G)=S\cup \{x\}\cup B$, where $|S|=n-3$ and $|B|=2$. Now in $G$, connect two vertices $s_1,s_2 \in S$ to $x$ and connect $s_1$ to each vertex in $B$; $G$ has $4$ edges and it can be easily verified that $CR(G)=1$ and $\gamma_{CR}(G)=n-3$. Now let $b_1 \in B$ and connect $x$ to $b_1$; $G$ now has $5$ edges and again it can be easily verified that $CR(G)=1$ and $\gamma_{CR}(G)=n-3$. Next, connect $b_1$ to the other vertex of $B$, say $b_2$, delete the edge $[s_1,b_2]$, and connect $s_2$ to $b_2$; $G$ has $6$ edges and again satisfies $CR(G)=1$ and $\gamma_{CR}(G)=n-3$. Finally, when $m > 6$, connect $m-6$ vertices in $S\setminus \{s_1,s_2\}$ to $x$; $G$ now has $m$ edges. Now, $x$ cannot be in any $\gamma_{CR}$-set of $G$, since then the only ways to dominate $b_2$ is if $b_1$, $s_2$ or $b_2$ are in the same $\gamma_{CR}$-set as $x$, and in each case more than one vertex will be overdominated. If $b_1$ is in some dominating set of $G$, then, in order to dominate $s_2$, either $x$, $s_2$, or $b_2$ must be in the same dominating set. In each case more than one vertex is overdominated. If $b_2$ is in some $\gamma_{CR}$-set, then by the above neither $x$ nor $b_1$ can be in the same set. Thus $S\setminus \{s_2\}$ must be in the same $\gamma_{CR}$-set as $b_2$, but $x$ and $b_1$ must both be overdominated. Since $S$ overdominates one vertex, we have that $CR(G)=1$, and $\gamma_{CR}(G)=n-3$. Therefore, $D(n,1,m)\geq n-3$. By Proposition \ref{prop1}, $D(n,1,m)\leq n-3$, hence, $D(n,1,m)=n-3$.\\

Case 2:  Assume $n+2 \leq m \leq {n-1 \choose 2}$ and ${n-(r+1) \choose 2} + (n-2) < m \leq {n-r \choose 2} + (n-2)$. 

Note this implies that $ 2 \leq r \leq n-4$. By Theorem \ref{Mn1r}, and since ${n - (r+1) \choose 2}+(n-2) < m$, we have that $D(n,1,m)\leq r$.

Let $G$ be a graph whose vertices are the disjoint union 
$V(G)=S\cup \{x\}\cup B$, where $|S|=r$ and $|B|=n-r-1$. Since $r\leq n-4$, $|B|\geq 3$. First consider the case where $r=2$, and let $S=\{s_1,s_2\}$. Connect $s_1$ and $s_2$ to $x$. Connect $s_1$ to each vertex in $B$ except for some vertex $b\in B$, and connect $s_2$ to $b$. Connect each pair of vertices in $B$ so that $B$ induces a complete graph. Note that $m-\left(\left(n-1\right)+{n-3 \choose 2}\right) < n-3$, therefore we can connect $x$ to $m-((n-1)+{n-3 \choose 2})$ vertices in $B$. We have that
\[
|E(G)|=2+(n-2-1)+{n-2-1 \choose 2}+\left(m-\left(\left(n-1\right)+{n-3 \choose 2}\right)\right)=m.
\]
Since the number of edges from $x$ to $B$ is strictly less than the number of vertices in $B$, we can arrange so that $x$ is not adjacent to $b$. Notice that $S$ is a dominating set that overdominates only $x$. If $x$ is in some $\gamma_{CR}$-set of $G$, then in order to dominate $b$, either $s_2$ or some other vertex $b'\in B$ must be in the same $\gamma_{CR}$-set as $x$. Since $s_2$ is adjacent to $x$, $s_2$ and $x$ cannot lie in the same $\gamma_{CR}$-set. For all $b'\in B$, $\{x,b'\}$ is a dominating set of $G$ that overdominates either $s_1$ or $s_2$. If for some $b'\in B$, $b'$ lies in a $\gamma_{CR}$-set of $G$ that does not contain $x$, then because $b'$ is adjacent to each vertex in $B$, no other vertex in $B$ is in the same $\gamma_{CR}$-set as $b'$, so $s_1$ or $s_2$ lies in the same $\gamma_{CR}$-set as $b'$. In both cases, however, at least one vertex is overdominated. Therefore, $CR(G)=1$ and $\gamma_{CR}(G)=2$, so $D(n,1,m)\geq 2$. Since $D(n,1,m)\leq 2$, $D(n,1,m)=2$. 

Now, consider the case where $r\geq3$. Redefine $G$ (keeping the same notation for the vertex set) as follows: connect each vertex in $S$ to $x$, connect a vertex $s_1\in S$ to every vertex in $B$ except one vertex $b_2\in B$, and connect a vertex $s_2 \in S$, $s_2\neq s_1$, to $b_2$. Connect vertices in $B$ so that $B$ induces a complete graph minus the edge $[b_1,b_2]$, and connect $x$ to a vertex $b_3\in B$, where $b_3\neq b_1,b_2$. We have that

\[
|E(G)|=r+(n-r-1)+\left({n-r-1 \choose 2}-1\right)+1={n-r-1 \choose 2}+(n-2)+1.
\]
Notice that $S$ is a dominating set that overdominates one vertex, $x$. If $x$ is in some $\gamma_{CR}$-set of $G$, then in order to dominate $b_2$, either $s_2$, $b_2$, or some $b'\in B$, $b'\neq b_1$ must be in the same $\gamma_{CR}$-set as $x$. However if $s_2$, $b_2$, or $b'$ is in the same $\gamma_{CR}$-set as $x$, then more than one vertex will be overdominated. Since $S$ overdominates only one vertex, $x$ cannot be contained in any $\gamma_{CR}$-set. Since each vertex in $V(G)\setminus \{x\}$ dominates exactly one vertex in $S$, every $\gamma_{CR}$-set of $G$ must contain at least $r$ elements. Since $r\geq3$, every $\gamma_{CR}$-set contains either two vertices in $S$ or two vertices in $B$. In the former case, $x$ is overdominated, and in the latter, $b_3$ is overdominated.

Hence, $CR(G)=1$ and $\gamma_{CR}(G)=r$.
Thus, in the case that $m={n-r-1 \choose 2}+(n-2)+1$, $D(n,1,m)\geq r$, so we can conclude that $D(n,1,m)=r$. 

Now, if $m > {n-(r+1) \choose 2}+(n-2)+1$, then connect $b_1$ to $b_2$ and connect $x$ to 
\[
m -\left({n-(r+1) \choose 2}+(n-2)+1\right)-1=m - {n-(r+1) \choose 2}-n
\]
vertices in $B\setminus \{b_3\}$. Since $m\leq {n-r \choose 2}+(n-2)$,
\[
m - {n-(r+1) \choose 2}-n\leq {n-r \choose 2}+(n-2) - {n-(r+1) \choose 2}-n=n-r-3,
\]
and because $|B\setminus \{b_3\}|=n-r-2$, it is always possible to arrange so that $x$ is not adjacent to $b_2$. Now, we have that
\begin{align*}
|E(G)|&=\left({n-(r+1) \choose 2}+(n-2)+1\right)\\ 
&\qquad +1+\left(m-\left({n-(r+1) \choose 2}+(n-2)+1\right)-1\right)\\
&=m.
\end{align*}

Notice that $S$ is a dominating set of $G$ that only overdominates one vertex, $x$. By a similar argument as above, $x$ cannot be in any $\gamma_{CR}$-set of $G$, and any $\gamma_{CR}$-set containing any element in $B$ must have at least $r$ elements and overdominate at least one vertex. Thus, $CR(G)=1$ and $\gamma_{CR}(G)=r$, so $D(n,1,m)\geq r$. Therefore $D(n,1,m)=r$.\\

Case 3:  Assume $m > {n-1 \choose 2}$.  

Recall by Proposition \ref{gammageq2} if $CR(G) = 1$, then $\gamma_{CR}(G) \geq 2$. The desired result now follows from Theorem \ref{Mn1r}.
\end{proof}

Now we will turn our attention to finding the value of $d(n, 1, m)$.  Note that if $CR(G) \geq 1$, then $\gamma_{CR}(G) \geq 2$.  This implies that if $d(n,1,m) \neq 0$, then $d(n,1,m) \geq 2$.

\begin{theorem} \label{dn1m}
If $n \geq 5$ then 
\[
d(n,1,m) = 
    \begin{cases}
        n-m+1 & \quad \text{if $4 \leq m \leq n-1$},\\
        2 & \quad \text{if $n \leq m \leq {n-1 \choose 2}$},\\
        0 & \quad \text{if $m > {n-1 \choose 2}$}.\\
    \end{cases}
\]
\end{theorem}

\begin{proof}
Assume that $n \geq 5$.\\

Case 1: Assume $4 \leq m \leq n-1$.

Let $r=n-m+1$. If $m\leq n-2$, then $r\geq 3$. If $d(n,1,m)=r' < r$, then, by Theorem \ref{mn1r}, $m(n,1,r')=n-r'+1\leq m=n-r+1$, which is a contradiction. Thus $d(n,1,m)\geq r$. For $m=n-1$, we have by Proposition \ref{gammageq2} that $d(n,1,m)\geq 2$. For any $4\leq m\leq n-1$, the graph $\Gamma$ from the proof of Theorem \ref{mn1r} satisfies $|V(\Gamma)|=n$, $|E(\Gamma)|=m$, $CR(\Gamma)=1$, and $\gamma_{CR}(\Gamma)=r$, so $d(n,1,m)\leq r$. Therefore, $d(n,1,m)=r$.\\

Case 2a:  Assume $n\leq m\leq 2n-5$.

Let $n\leq m\leq 2n-5$. Let $G$ be a graph whose vertices are the disjoint union $V(G)=S\cup \{x\}\cup B$ where $S=\{s_1,s_2\}$ and $|B|=n-3$. Connect $s_1$ and $s_2$ to $x$, connect $s_2$ to some vertex $b_2\in B$, and connect $s_1$ to each vertex in $B\setminus \{b_2\}$. Connect $b_2$ to $m-(n-1)$ vertices in $B$, notice $b_2$ is connected to at least one vertex in $B$. We have that
\[
|E(G)|=2+(n-3)+(m-(n-1))=m.
\]
Notice that $S$ is a dominating set of $G$ that overdominates only one vertex, $x$. If $x$ is contained in some $\gamma_{CR}$-set of $G$, then in order to dominate $b_2$, either $s_2$, $b_2$, or some $b\in B$, where $b_2$ is adjacent to $b$, is contained in the same $\gamma_{CR}$-set as $x$, in each of these cases, either $s_1$ or $s_2$ is overdominated and the $\gamma_{CR}$-set contains at least two elements. 

Let $S' \neq \{s_1, s_2\}$ be a $\gamma_{CR}$-set of $G$ that does not contain $x$. In order to dominate $x$, one of $s_1$ or $s_2$ must be in $S'$.  If $s_1 \in S'$ then $b_2 \in S'$. However, this implies there is a vertex $b$ in $B$ that is overdominated.  If instead, $s_2 \in S'$, then there is a $b \in B$ such that $b$ is adjacent to $s_1$ and $b \in S'$. If $b$ is adjacent to $b_2$, then there is at least one vertex that is overdominated ($b_2$).  If $b$ is not adjacent to $b_2$, then there is a vertex $b' \neq b_2$ in $B$ that is not adjacent to $s_2$ or $b$. In order to dominate $b'$, either $s_1$, $b_2$, or $b'$ must be in $S'$. In each case, at least one vertex is overdominated. 

Thus, every $\gamma_{CR}$-set of $G$ contains at least two elements and overdominates at least one vertex, so $CR(G)=1$ and $\gamma_{CR}(G)=2$. Therefore, $d(n,1,m)\leq 2$, and we may conclude that $d(n,1,m)=2$.\\

Case 2b: Assume $2n-4\leq m\leq {n-3 \choose 2}+(n-1)$.

Let $G$ be a graph whose vertices are the disjoint union $V(G)=S\cup \{x\}\cup B$, where $S=\{s_1,s_2\}$ and $|B|=n-3$. Connect $s_1$ and $s_2$ to $x$. Let $b_2$ be some vertex in $B$, and connect $s_1$ to each vertex contained in $B\setminus \{b_2\}$. Connect $s_2$ to $b_2$, and connect $b_2$ to each vertex contained in $B\setminus \{b_2\}$. Connect $m-(2n-5)$ edges between the vertices in $B\setminus \{b_2\}$ in any way, which is possible since $m -(2n-5) \leq {n-4 \choose 2}$.  We have that
\[
|E(G)|=2+(n-3)+(n-4)+(m-(2n-5))=m.
\]
Notice that $S$ is a dominating set of $G$ that overdominates only one vertex, $x$. If $x$ is contained in some $\gamma_{CR}$-set, then in order to dominate $b_2$, either $s_2$, $b_2$ or some $b\in B\setminus \{b_2\}$ must be contained in the same $\gamma_{CR}$-set. But, in each case respectively, $x$, $s_2$, or $s_1$ is overdominated. Let $S'\neq \{s_1,s_2\}$ be a $\gamma_{CR}$-set not containing $x$, and notice $s_1$ or $s_2$ is in $S'$ since $x$ must be dominated. If $s_1\in S
'$, then $b_2\in S'$ and each vertex in $B\setminus\{b_2\}$ is overdominated. If $s_2\in S'$, then some vertex $b\in B\setminus \{b_2\}$ is an element of $S'$, and $b_2$ is overdominated.

Thus, every $\gamma_{CR}$-set of $G$ contained at least two elements and overdominates at least one vertex, so $CR(G)=1$ and $\gamma_{CR}(G)=2$. Therefore, $d(n,1,m)\leq 2$, and we may conclude that $d(n,1,m)=2$.\\

Case 2c: Assume ${n-3 \choose 2}+n\leq m\leq {n-1 \choose 2}$.

Let $G$ be a graph whose vertices are the disjoint union $V(G)=S\cup \{x\}\cup B$, where $S=\{s_1,s_2\}$ and $|B|=n-3$. Connect $s_1$ and $s_2$ to $x$. Let $b_2$ be some vertex in $B$, and connect $s_1$ to each vertex contained in $B\setminus \{b_2\}$. Connect $s_2$ to $b_2$, and connect every pair of vertices in $B$ so that $B$ induces a complete graph. Connect $x$ to $m-({n-3 \choose 2}+(n-1))$ vertices in $B\setminus \{b_2\}$. We have that
\[
|E(G)|=2+(n-3)+{n-3 \choose 2}+\left(m-\left({n-3 \choose 2}+(n-1)\right)\right)=m.
\]
Notice that $S$ is a dominating set of $G$ that overdominates one vertex, $x$. If $x$ is in some $\gamma_{CR}$-set, then in order to dominate $b_2$, either $s_2$, or some $b\in B$ must be contained in the same $\gamma_{CR}$-set. But, in each case $x$, $s_1$ or $s_2$ is overdominated. If $S'\neq \{s_1,s_2\}$ is a $\gamma_{CR}$-set not containing $x$, then $S'$ contains a vertex from $S$ and a vertex from $B$ (since every two vertices in $B$ are adjacent). Therefore, at least one vertex in $B$ is overdominated by $S'$.

Thus, each $\gamma_{CR}$-set contains at least two elements and overdominates at least one vertex, so $CR(G)=1$ and $\gamma_{CR}(G)=2$. Therefore, $d(n,1,m)\leq 2$, and we may conclude that $d(n,1,m)=2$.\\

Case 3:  Assume $m > {n-1 \choose 2}.$

This case follows from a similar argument as in Case 3 of Theorem \ref{Dn1r}.  

\end{proof}

\section{Results for  $CR(G)=k$}

In this section, we find the maximum size of a graph $G$ such that $|V(G)|=n$ and $\gamma_{CR}(G)=2$. We will also establish that $M(n, k, 2)$ is an upper bound for $M(n, k, r)$.  

\begin{theorem}\label{Mnk2}
Let $n\geq 5$, $ k\leq n-2$ if $n$ is even, $ k\leq n-3$ if $n$ is odd then
\[
M(n,k,2) = {n-1 \choose 2} +\floor*{\frac{k}{2}}.
\]
\end{theorem}

\begin{proof}
Let $G$ be a graph of order $n$ such that $CR(G)=k$ and $\gamma_{CR}(G)=2$. If $k=0$ or $1$, then the result follows from Theorem \ref{M(n,0,r)} and Theorem \ref{Mn1r} respectively, so we will only concern ourselves with $k\geq 2$.   

Let $S$ be a $\gamma_{CR}$-set, $A$ be the overdominated vertices, and $B$ be the remaining vertices of $G$. If $S$ is independent, then $S\cap A= \emptyset$. The number of edges between $S$ and $A$ is at most $2k$, and the number of edges between $S$ and $B$ is precisely $n-k-2$. The subgraph induced by $A\cup B$ has at most ${n-2 \choose 2}$ edges, and because no vertex contained in $A$ can be connected to every other vertex, we must subtract at least $\ceil*{\frac{k}{2}}$ edges. Therefore,
\begin{align*}
|E(G)|&\leq \left(2k+(n-k-2)+{n-2 \choose 2}\right)-\ceil*{\frac{k}{2}}\\
&={n-1 \choose 2} +\floor*{\frac{k}{2}}.
\end{align*}

If the two vertices contained in $S$ are connected, then let $A'=A\setminus S$. Notice $|A'|=k-2$, and $|B|=n-k$. There are at most $2(k-2)$ edges between $S$ and $A'$, and precisely $n-k$ edges between $S$ and $B$. The subgraph induced by $A'\cup B$ has at most ${n-2 \choose 2}$ edges, and because no vertex contained in $A'$ can be connected to every other vertex, we must subtract at least $\ceil*{\frac{k-2}{2}}$ edges. Therefore,

\begin{align*}
|E(G)|&\leq \left(1+2(k-2)+(n-k)+{n-2 \choose 2}\right)-\ceil*{\frac{k-2}{2}}\\
&={n-1 \choose 2} +\floor*{\frac{k}{2}}.
\end{align*}

Therefore 
\[
M(n,k,2) \leq {n-1 \choose 2} + \floor*{\frac{k}{2}}.
\]

To establish the same lower bound for $M(n,k,2)$, we will consider two cases. \\

Case 1:  Assume $k$ is odd and $3 \leq k \leq n-3$

Let $G$ be a graph whose vertices are the disjoint union $S\cup A\cup B$, where $S=\{s_1,s_2\}$, $|A|=k-2$, and $|B|=n-k$ (notice that $|B|\geq 3$). Connect $s_1$ and $s_2$ to each vertex in $A$, connect $s_1$ to $s_2$, and connect $s_1$ to each vertex in $B\setminus \{b\}$ for some $b\in B$. Connect $s_2$ to $b$. Let $a\in A$ be arbitrary. Add edges so that $A\cup B$ induces a complete graph minus a $1$-factor of $A\setminus \{a\}$ and the edge $[a,b]$. We have that
\begin{align*}
   |E(G)|   &=1+2(k-2)+(n-k)+{n-2 \choose 2}-\frac{k-3}{2}-1\\
            & = {n-1 \choose 2} +\frac{k-1}{2}\\
            & ={n-1 \choose 2} +\floor*{\frac{k}{2}}.
\end{align*}

Notice that $S$ is a dominating set that overdominates precisely $k$ vertices. If some $a'\in A\setminus \{a\}$ is in some $\gamma_{CR}$-set, then in order to dominate the vertex in $A$ that $a'$ is not adjacent to, either $s_1$, $s_2$, some vertex in $A\setminus \{a'\}$ or some vertex in $B$ must be in the same $\gamma_{CR}$-set. In each of these cases, at least $k$ vertices are overdominated. If $a$ is in some $\gamma_{CR}$-set, then in order to dominate $b$, either $s_2$, some vertex in $B$, or some vertex in $A\setminus \{a\}$ must be in the same $\gamma_{CR}$-set. Again, in each of these cases, at least $k$ vertices are overdominated. If $S'\neq S$ is a $\gamma_{CR}$-set disjoint from $A$, then $S'$ contains either $\{s,b'\}$ or $\{b',b''\}$ for some $s\in S$ and $b',b''\in B$. Thus, $S'$ overdominates at least all of $A\setminus\{a\}$ and a vertex in $B$. Therefore, $CR(G)=k$ and $\gamma_{CR}(G)=2$, so $M(n,k,2)\geq {n-1 \choose 2} +\floor*{\frac{k}{2}}$.\\

Case 2:  Assume $k$ is even and $2 \leq k \leq n-2$.  

If $n$ is even and $k=n-2$, then the result follows by Theorem \ref{eventhm}, so assume $2\leq k\leq n-3$.

Let $G$ be a graph whose vertices are the disjoint union $V(G)=S\cup A\cup B$, where $S=\{s_1,s_2\}$, $|A|=k$, and $|B|=n-k-2$. Connect $s_1$ and $s_2$ to each vertex in $A$, and connect $s_1$ to each vertex in $B\setminus \{b\}$ for some $b\in B$. Connect $s_2$ to $b$. Add edges so that $A\cup B$ induces a complete graph minus a $1$-factor of $A$. We have that
\[
|E(G)|=2k+(n-k-2)+{n-2 \choose 2}-\frac{k}{2}={n-1 \choose 2} + \frac{k}{2}.
\]
Notice $S$ is a dominating set that overdominates only each vertex in $A$. If there exists some $a\in A$ that is in some $\gamma_{CR}$-set, then in order to dominate the vertex in $A$ that $a$ is not adjacent to, either $s_1$, $s_2$, some $a'\in A$, or some $b'\in B$ must be contained in the same $\gamma_{CR}$-set. In each of these cases, at least $k$ vertices are overdominated. Any $\gamma_{CR}$-set $S'\neq S$ disjoint from $A$ must contain one of the sets, $\{s, b'\}$ or $\{b', b''\}$ with $s \in S$ and $b', b'' \in B$, and will overdominate at least all of $A$ which contains $k$ vertices. Therefore, $CR(G)=k$ and $\gamma_{CR}(G)=2$, so $M(n,k,2)\geq {n-1 \choose 2} + \frac{k}{2}$.  

\end{proof}

\begin{theorem}\label{bbnd}
Let $n\geq 5$, $2\leq k\leq n-2$ if $n$ is even, $2\leq k \leq n-3$ if $n$ is odd. Let $G$ be a graph of order $n$ such that $CR(G)=k$ and $\gamma_{CR}(G)=r \geq 3$. If $G$ has a $\gamma_{CR}$-set, $S$, such that the subgraph induced by $S$ has $b\leq k$ non-isolated points then:
\[
    |E(G)| \leq \\
    \begin{cases}
        {b \choose 2} +(r-2)(k-b) +{n-r+1 \choose 2} + \floor*{\frac{k-b}{2}} & \text{ if $r \leq n-r-(k-b)$},\\
        {b \choose 2} +(r-2)(k-b)+{n-r+1 \choose 2}  & \text{ if $r > n-r-(k-b)$}.
    \end{cases}
\]
\end{theorem}

\begin{proof}
Let $A$ be the set of vertices disjoint from $S$ that are overdominated by $S$, and let $B$ be the set of remaining vertices of $G$ disjoint from $S$ and $A$. Notice that because each of the $b$ non-isolated vertices in the induced subgraph of $S$ are overdominated, $|A|=k-b$ and $|B|=n-r-(k-b)$. The number of edges in the subgraph induced by $S$ is at most ${b \choose 2}$, and the number of edges between $S$ and $B$ is precisely $n-r-(k-b)$.  The number of edges between $S$ and $A$ is at most $r(k-b)$, and the number of edges in the subgraph induced by $A\cup B$ is at most ${n-r \choose 2}$. Also, each vertex of $A$ is not adjacent to at least one other vertex. Therefore,
\begin{align*}
|E(G)|&\leq {b \choose 2}+(n-r-(k-b))+r(k-b)+{n-r \choose 2}- \ceil*{\frac{k-b}{2}}\\
&={b \choose 2} +(r-2)(k-b)+ {n-r+1 \choose 2} + \floor*{\frac{k-b}{2}}.
\end{align*}
If $n-r-(k-b) < r$, then $|B| < |S|$, so there is a vertex $s \in S$ that is not adjacent to any vertex in $B$.  If $\floor*{\frac{k-b}{2}} \neq 0$, then  $k-b \geq 2$ which implies that $|A|\geq 2$. Note that $deg(a) \leq n-2$ for all $a \in A$, otherwise $G$ has a dominating set of size 1.  Partition $A$ into two sets, $A_s$ and $A'$ where $A_s$ is the subset of vertices in $A$ that are adjacent to $s$ and $A'=A \setminus A_s$.  Assume $x \in A_s$.  Since $s$ is adjacent to only vertices that are overdominated by $S$, $|N[s]\cap N[x]|\leq k$. If $S \cup B \cup A' \subseteq N(x)$, then $\{x, s\}$ would form a dominating set of size two which overdominates a set of size at most $k$ contradicting that $\gamma_{CR}(G) \geq 3.$  Therefore, for every vertex in $x \in A_s$, there exists a vertex $x' \in S\cup B \cup A'$ so that $x$ is not adjacent to $x'$.  Additionally, every vertex in $A'$ is not adjacent to $s$.  This observation along with the relation stated in the beginning of the proof show 
\begin{align*}
|E(G)|&\leq {b \choose 2}+(n-r-(k-b))+r(k-b)+{n-r \choose 2}-(k-b)\\
&=(r-2)(k-b) +{n-r+1 \choose 2} + {b \choose 2}.
\end{align*}

\end{proof}

Now we can bound bound $M(n,k,r)$ above by $M(n,k,2)$ which we found in Theorem \ref{Mnk2}.

\begin{theorem}
Let $n\geq 5$, and let $ k\leq n-2$ if $n$ is even, $ k\leq n-3$ if $n$ is odd. Then $M(n,k,r)\leq M(n,k,2)$, where $r\geq 2$.
\end{theorem}

\begin{proof}
If $k=0$ or $1$, then the statement follows from Theorem \ref{M(n,0,r)} and Theorem \ref{Mn1r} respectively. 

Let $k\geq 2$ and $r\geq 3$. Let $G$ be a graph such that $CR(G)=k$ and $\gamma_{CR}(G)=r$. Let $S$ be a $\gamma_{CR}$-set of $G$, where the subgraph induced by $S$ has $b$ non-isolated vertices. We will show that 
\begin{align*}
|E(G)| &\leq (r-2)(k-b) +{n-r+1 \choose 2} +\floor*{\frac{k-b}{2}}+ {b \choose 2}\\
&\leq {n-1 \choose 2}+\floor*{\frac{k}{2}}\\
&=M(n,k,2).
\end{align*}
The first inequality follows from Theorem \ref{bbnd}, so we will show the second inequality.
Notice that 
\begin{align*}
&((n-r+1)+(n-r+2)+\cdots + (n-2)) + {n-r+1 \choose 2}\\
&\qquad \qquad  =(n-r)(r-2)+\frac{(r-1)(r-2)}{2} + {n-r+1 \choose 2}\\
&\qquad \qquad  ={n-1 \choose 2}.
\end{align*}
 
Notice that $k-b$ is the number of overdominated vertices of $G$ not contained in $S$ and $n-r$ is the number of vertices not contained in $S$. This implies that $k-b\leq n-r$. If $b\leq r-1$, then
\begin{align*}
&(r-2)(k-b) +{n-r+1 \choose 2} +\floor*{\frac{k-b}{2}}+ {b \choose 2}\\
& \qquad \qquad \leq (r-2)(n-r)+\frac{(r-1)(r-2)}{2} +{n-r+1 \choose 2} +\floor*{\frac{k}{2}}\\
&\qquad \qquad={n-1 \choose 2}+\floor*{\frac{k}{2}}.
\end{align*}
If $b=r$, then by Proposition \ref{gammabnd}, $k-b\leq n-r-1$, so
\begin{align*}
&(r-2)(k-b)+{b \choose 2} + {n-r+1 \choose 2} +\floor*{\frac{k-b}{2}} \\
& \qquad \qquad \leq (r-2)(n-r-1) +{r \choose 2} + {n-r+1 \choose 2} +\floor*{\frac{k}{2}} -1\\
& \qquad \qquad =(r-2)(n-r)+{r \choose 2}-(r-1)+{n-r+1 \choose 2} +\floor*{\frac{k}{2}}\\
& \qquad \qquad =(r-2)(n-r)+{r-1 \choose 2}+{n-r+1 \choose 2} +\floor*{\frac{k}{2}}\\
& \qquad \qquad ={n-1 \choose 2}+\floor*{\frac{k}{2}}.
\end{align*}
Since $G$ was arbitrary, $M(n,k,r)\leq M(n,k,2)$.

\end{proof}

Thus, the maximum value for $M(n,k,r)$ when $r$ can vary is $M(n,k,2)={n-1 \choose 2} + \floor*{\frac{k}{2}}$.

\section{Results for CR(G)=2}

Here we find the exact values of $M(n,2,r)$ and $D(n,2,m)$. First, we will give a bound for $\gamma_{CR}(G)$ when $CR(G)=2$.

\begin{proposition}\label{cr2bnd}
Let $G$ be a graph of order $n\geq 5$. If $CR(G)=2$, then $\gamma_{CR}(G)\leq n-2$.
\end{proposition}

\begin{proof}
If $G$ has a $\gamma_{CR}$-set with $n$ vertices, then $G$ has only one edge that connects the two overdominated vertices. However, there is no such $G$ with one edge such that $CR(G)=2$. If $G$ has a $\gamma_{CR}$-set with $n-1$ vertices, then there is at least one vertex in the $\gamma_{CR}$-set that is overdominated. This vertex must be adjacent to another vertex in the $\gamma_{CR}$-set, therefore there are two overdominated vertices in the $\gamma_{CR}$-set. This implies $G$ has two edges, one between the overdominated vertices in the $\gamma_{CR}$-set, and one connecting a vertex in the $\gamma_{CR}$-set the unique vertex not in the $\gamma_{CR}$-set. However, there is no graph $G$ with two edges such that $CR(G)=2$. Therefore, $\gamma_{CR}(G)\leq n-2$.

\end{proof}

Thus, we have $D(n,2,n)=d(n,2,n)=D(n,2,n-1)=d(n,2,n-1)=0$. In the following lemma, we classify the graphs $G$ such that $|V(G)|=n$, $CR(G)=2$, and $\gamma_{CR}(G)=n-2$.

\begin{lemma}\label{4cycle}
Let $n\geq 4$ then $G$ is graph with $n$ vertices such that $CR(G)=2$ and $\gamma_{CR}(G)=n-2$, if and only if $G$ is a four-cycle with $n-4$ isolated vertices.
\end{lemma}

\begin{proof}
Clearly, if $G$ is a four-cycle with $n-4$ isolated vertices, then $CR(G)=2$, and $\gamma_{CR}(G)=n-2$.

Let $S$ be a $\gamma_{CR}$-set of $G$, and let $A=\{a_1,a_2\}$ be the set of vertices overdominated by $S$. First assume that $A\subseteq S$. In this case $a_1$ and $a_2$ must be adjacent and no other two vertices of $S$ can be adjacent. Let $b_1$ and $b_2$ be the vertices contained in $V(G)\setminus S$. Both $b_1$ and $b_2$ are adjacent to exactly one vertex in $S$. If either $b_1$ or $b_2$ is adjacent to some vertex in $S\setminus A$, then in every possible case, $G$ is either a path or the disjoint union of paths (possibly unioned with isolated vertices), which has cardinality-redundance of $0$ in each case. Therefore, $b_1$ is adjacent to exactly one of $a_1$ or $a_2$, and $b_2$ is adjacent to the vertex in $A$ that $b_1$ is not adjacent to. If $b_1$ and $b_2$ are not adjacent, then $G$ is a path on four vertices unioned with $n-4$ isolated vertices, which has cardinality-redundance of $0$. Thus, $b_1$ and $b_2$ are adjacent, so $G$ is a four-cycle unioned with $n-4$ isolated vertices. 

Notice that it is not possible that $|A\cap S|=1$, so assume that $A\cap S=\emptyset$. If $a_1$ is adjacent to $a_2$ then $\{a_1\}\cup (S\setminus N(a_1))$ is a dominating set that overdominates at most one vertex, $a_2$. If $a_2$ is adjacent to some vertex in $S\setminus N(a_1)$, $\{a_1\}\cup (S\setminus N(a_1))$ is a dominating set that overdominates at most one vertex, $a_2$. The same argument for $a_1$ shows that $N(a_1) = N(a_2)$.  If $|N(a_1)| \geq 3$, then for any $s \in N(a_1)$, $\{a_1,s\}\cup S\setminus N(a_1)$ is a dominating set that overdominates two vertices, $a_1$ and $s$, but has cardinality less than $S$.  Therefore $|N(a_1)| = |N(a_2)| = 2$ which completes the proof. 

\end{proof}

\begin{lemma}\label{2lem}
Let $G$ be a graph of order $n\geq 5$ such that $CR(G)=2$ and $\gamma_{CR}(G)=n-3$. If $S$ is a $\gamma_{CR}$-set and $A$ is the set of overdominated vertices, then $|N[A]\cap S|=2$.
\end{lemma}

\begin{proof}
If $|A\cap S|=2$, then since $CR(G)=2$, the statement follows. Since it is not possible that $|A\cap S|=1$, we may assume that $A\cap S=\emptyset$. Let $A=\{a_1,a_2\}$ and let $b$ be the vertex such that $V(G)\setminus (S\cup A)=\{b\}$. Let $s\in S$ be the unique vertex of $S$ that is adjacent to $b$. Assume that some vertex in $A$, say $a_1$ without loss of generality, is adjacent to  at least three vertices in $S$. Consider the set $D=\{a_1\}\cup (S\setminus N(a_1))$. 

If $D$ dominates $a_2$ but not $b$, then $D\cup \{b\}$ is a dominating set that overdominates at most two vertices, $a_2$ and $s$, but has cardinality less than $S$. 

If $D$ dominates $b$ but not $a_2$, then $N(a_2) \cap S \subseteq N(a_1) \cap S$  Let $s' \in (N(a_2) \cap S)\setminus \{s\}$. If $s\notin D$ or $a_1$ is not adjacent to $b$, then $D \cup \{s'\}$ is a dominating set that overdominates two vertices, $s'$ and $a_1$, but has cardinality less than $S$. If $s\in D$ and $a_1$ is adjacent to $b$, then if $a_2$ is adjacent to $b$, $(D\setminus \{s\}) \cup \{b,\}$ is a dominating set that overdominates at most two vertices $a_1$ and $b$. If $a_2$ is not adjacent to $b$, then $(D\setminus \{s,a_1\})\cup \{a_2,b\}\cup (S\cap (N(a_1)\setminus N(a_2)))$ is a dominating set that overdominates at most one vertex, $a_1$.
  
If $D$ dominates $b$ and $a_2$, then $D$ is a dominating set that overdominates at most two vertices $b$ and $a_2$, but has cardinality less than $S$.

If $D$ does not dominate $b$ or $a_2$ and $s \in N(a_2) \cap S$, then $D \cup \{s\}$ is a dominating set that overdominates two vertices, $a_1$ and $s$.  If $D$ does not dominate $b$ or $a_2$ and $s \notin N(a_2) \cap S$, then $(S \setminus N(a_2)) \cup \{a_2\}$ is a dominating set that overdominates at most two vertices, $a_1$ and $b$.  Therefore, each vertex in $A$ is adjacent to precisely two vertices in $S$.

Let $A_1=N(a_1)\cap S$ and $A_2=N(a_2)\cap S$ with $|A_1| = |A_2|=2$. If  $|A_1 \cap A_2| \leq 1$, then let $D_i = \{a_i\} \cup (S \setminus A_i)$, where $i=1$ or $2$. Notice $D_1$ dominates $a_2$ and $D_2$ dominates $a_1$. If $D_i$ dominates $b$ for some $i$, then $D_i$ is a dominating set that overdominates at most two vertices, but has cardinality less than $S$. If neither $D_i$ dominates $b$, then $\{b\} = A_1\cap A_2$ and neither $a_1$ nor $a_2$ is adjacent to $b$. Then $(S\setminus \{s\}) \cup \{b\}$ is a dominating set that overdominates no vertices. Therefore, $|A_1\cap A_2|=2$, i.e., $|N[A]\cap S|=2$.

\end{proof}

The following theorem gives us the exact values for $M(n,2,r)$.

\begin{theorem}\label{Mn2r}
Let $n\geq 5$. Then
\[
M(n,2,r)=
    \begin{cases}
        {n-1 \choose 2}+1\quad &\text{if $r=2$},\\
        2(r-2)+{n-r+1 \choose 2}+1\quad & \text{if $3\leq r\leq n-4$ and  $n-r-2\geq r$},\\
        2(r-2)+{n-r+1 \choose 2}\quad &\text{if $3\leq r\leq n-4$ and $n-r-2<r$},\\
        7\quad & \text{if $r=n-3$},\\
        4\quad & \text{if $r=n-2$}.\\
    \end{cases}
\]

\end{theorem}

\begin{proof}

Case 1: Assume $r=2$.

It follows directly from Theorem \ref{Mnk2} that  $M(n,2,r)=M(n,2,2)={n-1 \choose 2}+1$.\\

Case 2: Assume $3\leq r\leq n-4$ and $n-r-2\geq r$.

In Theorem \ref{bbnd}, when $k=2$, the only possible values for $b$ are $0$ and $2$. The bounds in Theorem \ref{bbnd} in this case are maximized when $b=0$. Thus by Theorem \ref{bbnd}, we have that $M(n,2,r)\leq 2(r-2)+{n-r+1 \choose 2}+1$.

Let $G$ be a graph whose vertices are the disjoint union $V(G)=S\cup A\cup B$, where $|S|=r$, $|A|=2$, and $|B|=n-r-2$. Let $S=\{s_1,\dots,s_r\}$, $A=\{a_1,a_2\}$, and $B=\{b_1,\dots,b_{n-r-2}\}$. Note that $|B|\geq 2$.

Assume that $n-r-2\geq r$. Connect each vertex in $S$ to each vertex in $A$. Connect $s_i$ to $b_i$ for $1\leq i\leq r-1$, and connect $s_r$ to $b_j$ for all $j\geq r$. Connect each vertex in $A$ to each vertex in $B$, and connect every pair of vertices in $B$. Note $a_1$ and $a_2$ are not adjacent. We have that 
\[
|E(G)|=2r+(n-r-2)+{n-r \choose 2}-1=2(r-2)+{n-r+1 \choose 2}+1.
\]
Notice that $S$ is a dominating set of $G$ that overdominates two vertices $a_1$ and $a_2$. If $a_1$ is in some $\gamma_{CR}$-set, then in order to dominate $a_2$, either some vertex in $S$, some vertex in $B$, or $a_2$ must be in the same $\gamma_{CR}$-set. In each of these cases, more than two vertices are overdominated. Also, any dominating set containing $a_1$ and vertex in $S$ will overdominate more than two vertices. Let $S'\neq S$ be a $\gamma_{CR}$-set disjoint from $A$. Since each vertex in $V(G)\setminus A$ dominates precisely one vertex in $S$, $S'$ must contain at least $r$ elements. Since $S'$ contains either two elements in $B$ or two elements in $S$, $S'$ will overdominate at least two vertices. Therefore, $CR(G)=2$ and $\gamma_{CR}(G)=r$, so $M(n,2,r)\geq 2(r-2)+{n-r+1 \choose 2}+1$. This implies that $M(n,2,r) = 2(r-2)+{n-r+1 \choose 2}+1$.\\

Case 3: Assume $3\leq r\leq n-4$ and $n-r-2<r$.

Similar to Case 2, $M(n,2,r)\leq 2(r-2)+{n-r+1 \choose 2}$ by Theorem \ref{bbnd}.

Let $G$ be the graph that has the same vertex set as the graph in Case 2. Connect each vertex in $S$ to each vertex in $A$, connect $s_1$ to $b_1$, and connect $s_2$ to each vertex in $B\setminus \{b_1\}$. Connect $a_1$ to $a_2$, connect $a_1$ to each vertex in $B\setminus \{b_1\}$, and connect $a_2$ to each vertex in $B\setminus \{b_2\}$. Connect every pair of vertices in $B$. We now have that
\[
|E(G)|=2r+(n-r-2)+{n-r \choose 2}-2=2(r-2)+{n-r+1 \choose 2}.
\]
Notice that $S$ is a dominating set that overdominates two vertices, $a_1$ and $a_2$. If any vertex in $B$ and any vertex in $A$ are in the same dominating set, then at least one vertex in $B$, one vertex in $A$, and one vertex in $S$ are overdominated, so a $\gamma_{CR}$-set cannot contain a vertex in $A$ and a vertex in $B$. Also, any dominating set containing $a_1$ and vertex in $S$ will overdominate more than two vertices. Let $S'\neq S$ be a $\gamma_{CR}$-set disjoint from $A$. Since each vertex in $V(G)\setminus A$ dominates precisely one vertex in $S$, $S'$ must contain at least $r$ elements. Since $S'$ contains either two elements in $B$ or two elements in $S$, $S'$ will overdominate at least two vertices. Thus, every $\gamma_{CR}$-set contains at least $r$ vertices and overdominates at least two vertices. Therefore $CR(G)=2$ and $\gamma_{CR}(G)=r$, so $M(n,2,r)\geq 2(r-2)+{n-r+1 \choose 2}$. This implies that $M(n,2,r)=2(r-2)+{n-r+1 \choose 2}$.\\

Case 4: Assume $r=n-3$.

Let $G$ be a graph of order $n$ such that $CR(G)=2$ and $\gamma_{CR}(G)=n-3$. Let $S$ be a $\gamma_{CR}$-set of $G$, $A=\{a_1,a_2\}$ be the set of overdominated vertices, and $B$ be the remaining vertices. Either $A\cap S=A$ or $A\cap S=\emptyset$. If $A\cap S=A$, then $|B|=3$. There is one edge, $[a_1,a_2]$, between the vertices in $S$, three edges between $S$ and $B$, and at most ${3 \choose 2}=3$ edges between the vertices in $B$. Therefore, $|E(G)|\leq 1+3+3=7$. If $A\cap S=\emptyset$, then $|B|=1$. By Lemma \ref{2lem}, $|N[A]\cap S|=2$ so there are at most four edges between $S$ and $A$. There is one edge between $S$ and $B$, and there are at most ${3 \choose 2}=3$ edges in the subgraph induced by $A\cup B$. If a vertex $a\in A$ is adjacent to each vertex in $(A\cup B)\setminus\{a\}$, then $\{a\}\cup (S-(N[A]\cap S))$ is a dominating set that overdominates at most the one vertex in $B$. So we must subtract at least one edge to our count, thus, $|E(G)|\leq 4+1+3-1=7$. Therefore, $M(n,2,n-3)\leq 7$.

Now, consider the graph $\Gamma$ with five vertices, two of which, say $s_1$ and $s_2$, are connected by an edge, where $s_1$ is connected to two other vertices $b_1$ and $b_2$, $s_2$ is connected to another vertex $b_3$, and $\{b_1,b_2,b_3\}$ induces a complete graph. The graph $\Gamma'$ that is $\Gamma$ unioned with $n-5$ isolated vertices clearly has $n$ vertices, and it can be easily verified that $CR(\Gamma')=2$ and $\gamma_{CR}(\Gamma')=n-3$. Therefore, $M(n,2,n-3)\geq 7$.\\

Case 5: Assume $r=n-2$.

It follows directly from Lemma \ref{4cycle} that $M(n,2,n-2)=4$.\\

\end{proof}

The following two theorems combined give the exact values for $D(n,2,m)$. Theorem \ref{dn2msmall} considers the case when $4 \leq m \leq 2(n-6)+10$ and Theorem \ref{dn2mlarge} when $m > 2(n-6)+10$.

\begin{theorem}
\label{dn2msmall}
Let $n\geq 8$, then 

\[
D(n, 2, m) = 
\begin{cases}
    n-2  &\text{ if $m = 4$},\\
    n-3  &\text{ if $5 \leq m \leq 7$},\\
    n-4  &\text{ if $8 \leq m \leq 2(n-6)+10$}.
\end{cases}
\]
\end{theorem}

\begin{proof}
Case 1: Assume $m=4$.

It follows directly from Proposition \ref{cr2bnd} and Lemma \ref{4cycle} that $D(n,2,m)=n-2$.\\

Case 2: Assume $5 \leq m \leq 7$.

Since $m> 4$, we have by Theorem \ref{Mn2r} that $D(n,2,m)\leq n-3$. The graph that is a disjoint union of a four-cycle, a path on two vertices, and $n-6$ isolated vertices has cardinality-redundance of two, a $\gamma_{CR}$-set of size $n-3$, and five edges. Thus, $D(n,2,5)\geq n-3$. 

Consider the following graph of order $n$ where five of its vertices $x_i$, $1\leq i\leq 5$, are connected in the following way: $x_1$ is connected to $x_2$, $x_2$ is connected $x_3$, $x_3$ is connected to $x_4$, $x_4$ is connected to $x_1$, $x_1$ is connected to $x_5$, and $x_5$ is connected to $x_3$, and the remaining $n-5$ vertices are isolated. This graph can easily be verified to have cardinality-redundance of two and a $\gamma_{CR}$-set of size $n-3$. Since this graph has six edges, $D(n,2,6)\geq n-3$.

Consider a graph of order $n$ where five of its vertices $x_i$ $1 \leq i \leq 5$ are connected in the following way: $x_1$ is connected to $x_2$ and $x_5$, $x_2$ is connected to $x_3$ and $x_4$, $x_3$ is connected to $x_4$ and $x_5$, $x_4$ is connected to $x_5$, and the remaining $n-5$ vertices are isolated.  This graph has seven edges, cardinality-redundance of two, and has a $\gamma_{CR}$-set of size $n-3$. Therefore, $D(n,2,7)\geq n-3$.

By the above arguments we have the desired result.\\

Case 3: Assume $8 \leq m \leq 2(n-6)+10$.

Since $m>7$, by Theorem \ref{Mn2r}, $D(n,2,m)\leq n-4$. For $m=8$ and $9$, we refer to the graphs $G_1$ and $G_2$ below, respectively.

\begin{align*}
&\begin{tikzpicture}
\filldraw[black] (0,0) circle (2pt) node[anchor=east] {$x_1$};
\filldraw[black] (0,-1.4) circle (2pt) node[anchor=east] {$x_4$};
\filldraw[black] (1.4,0) circle (2pt) node[anchor=west] {$x_2$};
\filldraw[black] (1.4,-1.4) circle (2pt) node[anchor=west] {$x_3$};
\filldraw[black] (-1,-2.4) circle (2pt) node[anchor=north] {$x_6$};
\filldraw[black] (0.7,-0.7) circle (2pt) node[anchor=west] {$x_5$};
\draw (0,0) .. controls +(up:0cm) and +(left:0cm) .. (1.4,0);
\draw (0,0) .. controls +(up:0cm) and +(left:0cm) .. (0,-1.4);
\draw (1.4,0) .. controls +(up:0cm) and +(left:0cm) .. (1.4,-1.4);
\draw (1.4,-1.4) .. controls +(up:0cm) and +(left:0cm) .. (0,-1.4);
\draw (0,-1.4) .. controls +(up:0cm) and +(left:0cm) .. (-1,-2.4);
\draw (0,0) .. controls +(up:0cm) and +(left:0cm) .. (0.7,-0.7);
\draw (0,-1.4) .. controls +(up:0cm) and +(left:0cm) .. (0.7,-0.7);
\draw (1.4,-1.4) .. controls +(up:0cm) and +(left:0cm) .. (0.7,-0.7);
\end{tikzpicture}
&\begin{tikzpicture}
\filldraw[black] (0,0) circle (2pt) node[anchor=east] {$s_1$};
\filldraw[black] (0,-1.4) circle (2pt) node[anchor=east] {$s_2$};
\filldraw[black] (1.4,0) circle (2pt) node[anchor=north] {$a_1$};
\filldraw[black] (1.4,-1.4) circle (2pt) node[anchor=south] {$a_2$};
\filldraw[black] (2.8,0) circle (2pt) node[anchor=west] {$b_1$};
\filldraw[black] (2.8,-1.4) circle (2pt) node[anchor=west] {$b_2$};
\draw (0,0) .. controls +(up:0cm) and +(left:0cm) .. (1.4,0);
\draw (0,0) .. controls +(up:0cm) and +(left:0cm) .. (1.4,-1.4);
\draw (0,-1.4) .. controls +(up:0cm) and +(left:0cm) .. (1.4,-1.4);
\draw (0,-1.4) .. controls +(up:0cm) and +(left:0cm) .. (1.4,0);
\draw (1.4,0) .. controls +(up:0cm) and +(left:0cm) .. (2.8,0);
\draw (2.8,0) .. controls +(up:0cm) and +(left:0cm) .. (2.8,-1.4);
\draw (1.4,-1.4) .. controls +(up:0cm) and +(left:0cm) .. (2.8,-1.4);
\draw (0,0) .. controls +(up:1cm) and +(right:0cm) .. (2.8,0);
\draw (0,-1.4) .. controls +(up:-1cm) and +(left:0cm) .. (2.8,-1.4);
\end{tikzpicture}\\
&\ \ \ \ \ \ \ \ \ \ \ \ \ \ \ \ \ \ \ \ \ \ G_1 &G_2\ \ \ \ \ \ \ \ \ \ \ \ \ \ \ \ \ \ \  
\end{align*}
\begin{center}
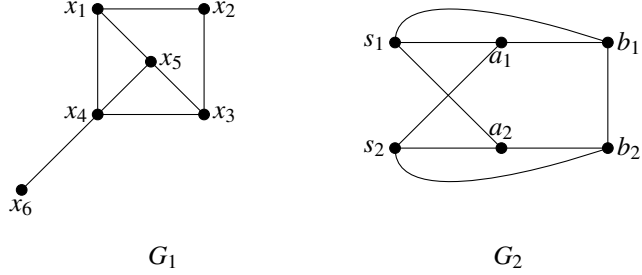
Figure 1: Graphs $G_1$ and $G_2$.
\end{center}
Notice that any dominating set of $G_1$ must contain either $x_4$ or $x_6$. Any dominating set containing $x_4$ must contain either $x_1$, $x_2$, or $x_3$ in order to dominate $x_2$. In each case the dominating set will overdominate at least two vertices. In particular $\{x_4,x_2\}$ is a dominating set that overdominates two vertices. Since there is no vertex that dominates $x_1$, $x_2$, $x_3$, and $x_5$, two vertices among the set $\{x_1,x_2,x_3,x_4,x_5\}$ must be in any dominating set containing $x_6$. In each case, at least two vertices are overdominated. Therefore, $CR(G_1)=2$ and $\gamma_{CR}(G_1)=2$. The graph obtained by taking the union of $G_1$ with $n-6$ isolated vertices has cardinality-redundance of two, a $\gamma_{CR}$-set of size $n-4$, and eight edges. Thus, $D(n,2,8)\geq n-4$.

Notice that $\{s_1,s_2\}$ is a dominating set of $G_2$ that overdominates two vertices. By inspection, every two vertices of $G_2$ overdominates at least two vertices. Since every dominating set of $G_2$ contains at least two vertices, $CR(G_2)=2$ and $\gamma_{CR}(G_2)=2$. The graph obtained by taking the disjoint union of $G_2$ with $n-6$ isolated vertices has cardinality-redundance of two, a $\gamma_{CR}$-set of size $n-4$, and nine edges. Thus, $D(n,2,9)\geq n-4$.

Let $m=10+2t+i\leq 2(n-6)+10$ for some $t$ where $i=0$ or $1$. Consider the graph $G$ whose vertices are the disjoint union $S\cup A\cup B$ where $S=\{s_1,\dots,s_{n-4}\}$, $A=\{a_1,a_2\}$, and $B=\{b_1,b_2\}$. Connect $s_1$ and $s_2$ to both $a_1$ and $a_2$. Connect $s_1$ to $b_1$ and connect $s_2$ to $b_2$. Connect $a_1$ to $b_1$, connect $a_2$ to $b_2$, connect $b_1$ to $b_2$, and connect $a_1$ to $a_2$. If $i=0$, then for all $3\leq j\leq t+2$, connect $s_j$ to $a_1$ and $a_2$. If $i=1$, then for all $3\leq j\leq t+2$, connect $s_j$ to $a_1$ and $a_2$, and connect $s_{t+3}$ to $a_1$. We have that
\[
|E(G)|=10+2t+i.
\]
Notice that $S$ is a dominating set that overdominates two vertices, $a_1$ and $a_2$. If some vertex in $A$, say $a_1$ without loss of generality, is in some dominating set, then in order to dominate $b_2$, either $b_1$, $b_2$, or $s_2$ must be in the same dominating set, but in each of these cases more than two vertices are overdominated. Thus, no vertex in $A$ is contained in any $\gamma_{CR}$-set. Since each vertex in $V(G)\setminus A$ dominates exactly one vertex in $S$, in order to dominate each vertex in $S$, each $\gamma_{CR}$-set must contain at least $n-4$ elements. If a $\gamma_{CR}$-set contains an element in $B$, say $b_1$, then the $\gamma_{CR}$-set must include $S\setminus \{s_1\}$ or $(S\setminus \{s_1,s_2\})\cup \{b_2\}$ and thus overdominates at least two vertices.  Therefore, $CR(G)=2$ and $\gamma_{CR}(G)=n-4$, so $D(n,2,m)\geq n-4$. By the above, we have $D(n,2,m)=n-4$ when $7<m\leq 2(n-6)+10$.
\end{proof}

\begin{theorem}
\label{dn2mlarge}
Let $n\geq 8$ and $m > 2(n-6)+10$.  For any positive integer $j$ define $A(j) = 2\left(j-2\right) + {n-j+1 \choose 2}.$ Then 
\[D(n,2,m) = r\]
where  

\begin{align*}
    & & A(r+1)+1 < &m \leq A(r)+1 &\text{and} & &     2 \leq &r < \floor*{\frac{n-2}{2}},& \\
   \text{or } & & A(r+1) < &m \leq A(r)+1 &\text{and} & &  &r=\floor*{\frac{n-2}{2}},\\
    \text{or }& &  A(r+1) < &m \leq A(r) &\text{and} & &   \floor*{\frac{n-2}{2}} < &r \leq n-5.\\
\end{align*}

\end{theorem}
\begin{proof}
Case 1: Assume $A(r+1)+1 < m \leq A(r)+1$ and  $ 2 \leq r < \floor*{\frac{n-2}{2}}$.

Because $2 \leq r < \floor*{\frac{n-2}{2}}$, $n-r-2\geq r$ and $n-(r+1)-2\geq r+1$. By Theorem \ref{Mn2r}, $D(n,2,m)\leq r$.

Consider the graph $G$ whose vertices are the disjoint union $V(G)=S\cup A\cup B$ where $S=\{s_1,\dots,s_r\}$, $A=\{a_1,a_2\}$, and $B=\{b_1,\dots,b_{n-r-2}\}$. Connect each vertex in $S$ to each vertex in $A$, connect $s_i$ to $b_i$ for $1\leq i\leq r-1$, and connect $s_r$ to $b_j$ for $j\geq r$. Connect each vertex in $A$ to each vertex in $B$. Add $m-(2r+3(n-r-2))$ edges between the vertices of $B$ in any way. We have that
\[
|E(G)|=(2r+3(n-r-2))+(m-(2r+3(n-r-2)))=m.
\]
Notice that $S$ is a dominating set that overdominates two vertices $a_1$ and $a_2$.
If some vertex in $A$, say $a_1$ is in a dominating set, then in order to dominate $a_2$, some vertex in $V(G)\setminus \{a_1\}$ must be in the same dominating set, but then more than two vertices are overdominated. Thus, no vertex in $A$ can be in any $\gamma_{CR}$-set. Since each vertex in $S\setminus A$ dominates exactly one vertex in $S$, every $\gamma_{CR}$-set contains at least $r$ vertices. Because each vertex in $S\setminus A$ dominates both $a_1$ and $a_2$, every $\gamma_{CR}$-set dominates at least two vertices. Therefore, $CR(G)=2$ and $\gamma_{CR}(G)=r$, so $D(n,2,m)\geq r$. By the above, $D(n,2,m)=r$.\\

Case 2: Assume $A(r+1) < m \leq A(r)+1$ and $r=\floor*{\frac{n-2}{2}}$.

Since $r=\floor*{\frac{n-2}{2}}$, $n-r-2\geq r$ and $n-(r+1)-2 < r+1$. Since $r\geq 3$, $D(n,2,m)\leq r$ by Theorem \ref{Mn2r}.

The same construction and arguments used in Case 1 can be applied to show that $D(n,2,m)\geq r$. Thus, $D(n,2,m)=r$.\\

Case 3: Assume $A(r+1) < m \leq A(r)$ and $\floor*{\frac{n-2}{2}} < r \leq n-5$.

Notice that $\floor*{\frac{n-2}{2}}<r$ if and only if $n-r-2<r$. Also, since $2\leq \floor*{\frac{n-2}{2}}<r$, by Theorem \ref{Mn2r}, $D(n,2,m)\leq r$.

Consider the graph $G$ whose vertices are the disjoint union $V(G)=S\cup A\cup B$ where $S=\{s_1,\dots,s_r\}$, $A=\{a_1,a_2\}$, and $B=\{b_1,\dots,b_{n-r-2}\}$. Connect each vertex in $S$ to each vertex in $A$, connect $s_1$ to $b_1$, and connect $s_2$ to each vertex in $B\setminus \{b_1\}$. Connect $a_1$ to each vertex in $B\setminus \{b_1\}$, connect $a_2$ to each vertex in $B\setminus \{b_2\}$, and connect $a_1$ to $a_2$. Connect $m-(2r+3(n-r-2)-1)$ edges between the vertices in $B$ in any way. We have that
\begin{align*}
|E(G)|&=2r+(n-r-2)+2(n-r-3)+1+(m-(2r+3(n-r-2)-1))\\
&=(2r+3(n-r-2)-1)+(m-(2r+3(n-r-2)-1))\\
&=m.
\end{align*}
Notice that $S$ is a dominating set that overdominates two vertices $a_1$ and $a_2$. If some vertex in $A$, say $a_1$, is in some dominating set, then in order to dominate $b_1$, either $a_2$, $s_1$, or some vertex in $B$ must be in the same dominating set, but in each case, more than two vertices are overdominated. A similar argument applies to $a_2$, so no vertex in $A$ can be in a $\gamma_{CR}$-set. Since each vertex in $S\setminus A$ dominates exactly one vertex in $S$, every $\gamma_{CR}$-set contains at least $r$ vertices. Since $r\geq3$, a $\gamma_{CR}$-set $S'$ disjoint from $A$ must contain either two vertices in $S$ or two vertices in $B$. In both cases, at least two vertices are overdominated.

Therefore, $CR(G)=2$ and $\gamma_{CR}(G)=r$, so $D(n,2,m)\geq r$. By the above, $D(n,2,m)=r$.
\end{proof}

\section{Acknowledgements}

These results are based upon work supported by the National Science Foundation under the grant number DMS-1560019. We would especially like to thank Eugene Fiorini and Byungchul Cha at Muhlenberg College for making all of this possible.

\end{document}